\documentclass[12pt]{amsart}
\usepackage[T1]{fontenc}
\usepackage{stix}
\usepackage{bm}
\usepackage{amsthm}
\usepackage{verbatim}
\usepackage{xcolor}
\usepackage[margin=1.3in]{geometry}

\usepackage[
	colorlinks,
	linkcolor={black},
	citecolor={black},
	urlcolor={black}
]{hyperref}

\newtheorem{theorem}{Theorem}[section]

\newtheorem{lemma}[theorem]{Lemma}

\theoremstyle{definition}
\newtheorem{definition}[theorem]{Definition}
\newtheorem{remark}[theorem]{Remark}

  \newtheorem{claim}{Claim}
  
\newenvironment{claimproof}[1][Proof of Claim]{\noindent \underline{#1.} }{\hfill$\diamondsuit$}

\newcommand{\bbC}{\mathbb{C}}

\newcommand{\bbN}{\mathbb{N}}
\newcommand{\bbR}{\mathbb{R}}
\newcommand{\bbZ}{\mathbb{Z}}

\newcommand{\scrF}{{\mathscr{F}}}

\newcommand{\SL}{\mathrm{SL}}
\newcommand{\SO}{\mathrm{SO}}

\DeclareMathOperator{\Leb}{Leb}

\newcommand{\great}{gap-rich}
\newcommand{\pregreat}{non-ahyperbolic}
\newcommand{\ample}{gap-rich}
\newcommand{\energy}{E}

\DeclareMathOperator{\tr}{trace}
\DeclareMathOperator{\lcm}{lcm}
\DeclareMathOperator{\per}{Per}
\DeclareMathOperator{\LP}{LP}
\DeclareMathOperator{\spectrum}{spec}
\DeclareMathOperator{\resolvent}{resolv}
\newcommand{\idty}{\mathbb{1}}

\newcommand{\lyap}{{\rm Lyap}}
\DeclareMathOperator{\spr}{spr}
\newcommand{\ids}{{\rm IDS}}
\DeclareMathOperator{\imaginary}{Im}
\renewcommand{\Im}{\imaginary}
\DeclareMathOperator{\real}{Re}
\renewcommand{\Re}{\real}
\DeclareMathOperator{\cyclic}{cyc}
\newcommand{\concatenate}{\sharp}
\newcommand{\aggregate}{\natural}
\newcommand{\Hausdorff}{{\rm H}}
\newcommand{\liesl}{\operatorname{sl}}
\DeclareMathOperator{\interior}{int}

\newcommand{\set}[1]{\left\{ #1 \right\}}

\numberwithin{equation}{section}

\title[Thin Spectra for Word Models]{Thin Spectra for Periodic and Ergodic Word Models}
\author[J.\ Fillman]{Jake Fillman}
\email{\href{mailto:fillman@tamu.edu}{fillman@tamu.edu}}
\address{Department of Mathematics, Texas A\&M University, College Station, TX  77843-3368}

\author[M. N. Gradner]{Michala N. Gradner}
\address{Department of Mathematics, Texas State University, San Marcos, TX  78666}
\author[H. J. Hendricks]{Hannah J. Hendricks}
\address{Department of Mathematics, Texas State University, San Marcos, TX  78666}
\date{}

\begin{document}

\begin{abstract}
    We establish a new and simple criterion that suffices to generate many spectral gaps for periodic word models.
    This leads to new examples of ergodic Schr\"odinger operators with Cantor spectra having zero Hausdorff dimension that simultaneously may have arbitrarily small supremum norm together with arbitrarily long runs on which the potential vanishes.
\end{abstract}

\maketitle

\hypersetup{linkcolor={black!30!blue}, citecolor={black!30!green},urlcolor={black!30!blue}}

\section{Introduction}

This paper is about linear operators with limit-periodic coefficients, which have been studied extensively over the years; we point the reader to \cite{DF2020LP} and \cite[Chapter~8]{DF2024ESO2} for more background and history. 
A function is called \emph{limit-periodic} if it lies in the  closure of the set of periodic functions in the uniform  topology. 
Avila proved that generic discrete one-dimensional Schr\"odinger operators with limit-periodic potentials have  purely singular continuous spectrum supported on a Cantor set of zero Hausdorff dimension \cite{Avila2009CMP}. 
Afterwards, this phenomenon and the mechanisms leading to it have been shown to be robust, occurring in many other scenarios such as quantum walks, continuum Schr\"odinger operators, Dirac operators, and  certain graph Laplacians.

The main advance of the recent work \cite{EFGL2022JFA} is to produce mechanisms that lead to thin spectrum in terms of simple group-theoretic criteria, such as the noncommutation of suitable monodromy matrices.
These methods %, which are inspired by recent proofs of positive Lyapunov exponents for certain families of random operators \cite{BDFGVWZ2019JFA}, 
work in quite a high level of generality.
However, the approach of \cite{EFGL2022JFA} can only produce conclusions of the following type: given some ``piece'' of potential, there is another nearby ``piece'' yielding the desired noncommutation.
 Crucially, one needs ``non-commutation at every scale'' and the relevant statements are (generally) supplied by inverse spectral theory.
Thus, as soon as one passes to a restricted set of possible potential values or to a particular rule dictating local correlations in a potential, the set of allowable pieces lies in a proper subset of possible realizations, and the approach no longer yields fruitful conclusions.
The present work introduces a brand-new criterion that is even simpler than that of \cite{EFGL2022JFA} and which furthermore produces new results in situations in which the approach of \cite{EFGL2022JFA} is not applicable or for which it is overly complicated to implement.

The present work circumvents the need to explicitly verify noncommutation at all scales.
At a high level, the main criterion says that as long as (away from a discrete exceptional set of energies), the periodic monodromy matrices (as a function of a suitable parameter) are not \emph{always} nonhyperbolic, then, upon passing to higher scales, one can \emph{perturbatively} open spectral gaps wherever one pleases away from the exceptional set.
In particular, we emphasize that the criterion only needs to be checked at a single scale.
The new criterion is strong enough that it suffices to establish the presence of thin Cantor spectrum, and we show that it applies in previous cases of interest as well as new ones.
We give more precise statements below.

As a sample application of our work, we include the following striking result: for generic limit-periodic potentials, the potential obtained by inserting a string of $(k-1)$ zeros between each entry of the potential has spectrum of vanishing Hausdorff dimension for all choices of $k$ and any choice of coupling constant.
More precisely, let $\LP(\bbZ):= \LP(\bbZ,\bbR)$ denote the set of limit-periodic functions $\bbZ \to \bbR$. For any $V:\bbZ \to \bbR$ and $k \in \bbN$, let us denote by $V^{[k]}$ the $k$-fold \emph{sieve} of $V$ obtained by inserting $k-1$ consecutive zeros between each entry of $V$, that is:
\begin{equation} \label{eq:main:sieve}
    V^{[k]}(n)
    = \begin{cases}
        V(n/k) & n \in k\bbZ, \\
        0 & \text{otherwise.}
    \end{cases}
\end{equation}
For $W:\bbZ \to \bbR$ bounded, we denote by $H_W = \Delta+W:\ell^2(\bbZ) \to \ell^2(\bbZ)$ the associated \emph{discrete Schr\"odinger operator}, which acts on $\psi \in \ell^2(\bbZ)$ via:
\begin{equation}
    [H_W\psi](n)
    = \psi(n-1) + \psi(n+1) + W(n)\psi(n).
\end{equation}

\begin{theorem} \label{t:main:sieve}
For generic limit-periodic $V:\bbZ \to \bbR$, $\spectrum(H_{\lambda V^{[k]}})$ is a Cantor set of zero Hausdorff dimension  for all $\lambda \neq  0$ and all $k \in \bbN$.
\end{theorem}

\begin{remark}
\mbox{\,}
\begin{enumerate}
\item[(a)]    The result is striking precisely because uniform approximation by periodic potentials, the insertion of long strings of zeros, and taking $\lambda>0$ small put the operator ``close'' to the free operator both in a uniform (that is $\ell^\infty$) sense, but also in a strong local sense: the density of values at which the potential differs from zero (i.e., the potential energy of the free operator) goes to zero as $k \to \infty$. 
All aspects conspire together to make the spectrum thicker, but their collective effort nevertheless fails.

\item[(b)]    To the best of the authors' knowledge, this gives the first examples of almost-periodic potentials $V \in \ell^\infty(\bbZ)$ with the property that $\spectrum(H_V)$ has Cantor spectrum of zero Hausdorff dimension which is persistent after both the insertion of strings of zeros of arbitrarily high density in $\bbZ$ and the introduction of a coupling constant tending to zero.

\item[(c)] Notice that $\spectrum H_V$ having zero Hausdorff dimension implies that it has zero Lebesgue measure as well, and hence the spectral type is purely singular.
By Gordon-type arguments  \cite{Gordon1976}, $\Delta + \lambda V^{[k]}$ has purely continuous spectrum for generic $V \in \LP(\bbZ)$, all $\lambda \neq 0$, and all $k \in \bbN$.
Thus, by the Baire category theorem, one can upgrade the conclusion of Theorem~\ref{t:main:sieve} to purely singular continuous spectrum of zero Hausdorff dimension for generic $V$, all $\lambda \neq 0$, and all $k \in \bbN$.
Suitable Gordon criteria can be found, e.g., in \cite[Section~7.8]{DF2022ESO}. 
Additionally, arithmetic versions can be found in the literature; see \cite{AJ2009Annals, AYZ2017Duke, liu2025newproofsharpgordons}.

\item[(d)] One can fix $\bm b = (b_1,\ldots, b_{k-1})$ and insert this string between each entry of $V$, defining $V^{[k,\bm b]}(n) = V(n/k)$ for $n \in k \bbZ$ and $V(n) = b_{n \operatorname{mod} k}$ for $n$ not divisible by $k$ . 
Then, $H_{V^{[k,\bm b]}}$ has zero-dimensional spectrum for generic $V \in \LP$.
Compare Theorem~\ref{t:prerichSchroExampleSieving}.

\item[(e)] These results are inspired by and of a piece with work on \emph{random word models} (compare \cite{DamSimSto2004JFA, deBievreGerminet2000JSP, JitSimSto2003CMP, Rangamani2022AHP}) in the sense that there is an underlying mechanism (limit-periodicity in our case, randomness in the other case) that is so strong that the spectral output (Anderson localization in the random word case, thin Cantor spectrum in the present case) persists for word models that are generated using the relevant mechanism.
The most closely related model (indeed, one of the inspirations of this work) is the \emph{trimmed Anderson model} \cite{ElgartSodin2014JST, ElgartSodin2017JST}, in which one only inserts random variables at a sublattice of points in $\bbZ$ (more generally on a sublattice of a suitable lattice).

\item[(f)] It would be of interest to better understand the consequences of the sieving process for discrete Schr\"odinger operators, especially ergodic operators, since it can naturally be encoded via a product dynamical system in that setting. 
For instance, if $V$ is an aperiodic element of a Boshernitzan subshift over a finite alphabet $\mathcal{A} \subseteq \bbR$, then $\spectrum(H_{V^{[k]}})$ is a Cantor set of zero Lebesgue measure for all $k \in \bbN$, which follows from the results of \cite{DamFilGoh2022JST}.
In the setting of the Fibonacci subshift,  it can be seen that the dimension of the spectrum can be altered by sieving \cite{FillmanLuna}.
Related work in the setting of quasi-periodic operators includes the \emph{mosaic model}, see \cite{WangEtAl2023CMP, ZhouEtAl2023PRL} and references therein.

\item[(g)] As in \cite{DamFilGor2019AHP}, one can strengthen Hausdorff dimension zero to lower box-counting dimension zero on compact subsets of the complement of a discrete exceptional set of energies and hence can obtain spectral results for higher-dimensional operators on $\ell^2(\bbZ^d)$.
Here there is a small subtlety; the union of countably many sets of Hausdorff dimension zero has Hausdorff dimension zero, but the same conclusion cannot be said for the box-counting dimension.
Thus, by exhausting the complement of the exceptional set by compact sets, we can draw conclusions regarding the Hausdorff dimension but not the box-counting dimension of the whole spectrum.
\end{enumerate}
\end{remark}

One could also repeat each value of the potential $k$ times consecutively, which is implemented by
\begin{equation}
V^{\circledast k}(n) = V(\lfloor n/k \rfloor), \quad n \in \bbZ.
\end{equation}
Again our methods yield generic Cantor spectrum of zero Hausdorff dimension.

\begin{theorem} \label{t:main:repeat}
For generic limit-periodic $V:\bbZ \to \bbR$, $\spectrum(H_{\lambda V^{\circledast k}})$ is a Cantor set of zero Hausdorff dimension for all $\lambda \neq  0$ and all $k \in \bbN$.
\end{theorem}

To illustrate the breadth of the approach, we complement the results for discrete Schr\"odinger operators with similar results for continuum Schr\"odinger operators, which follow as a consequence from the same framework.
To that end, recall that a continuum Schr\"odinger operator in $L^2(\bbR)$ with bounded potential $V:\bbR \to \bbR$ is given by
\begin{equation}
L_V\psi = -\psi''  + V \psi
\end{equation}
on a suitable domain of self-adjointness.
In order to define the relevant examples, we will need some more notation, so we postpone the statements of those theorems for a moment; see Section~\ref{ssec:continuumops}.

The rest of the paper is laid out as follows.
We introduce the framework in which we work and the main abstract criteria in Section~\ref{sec:results}.
There are two criteria formulated at the level of suitable one-parameter maps of a complete metric space into $\SL(2,\bbR)$: a map is \great\ if one is able to perturb any chain of iterates and produce a hyperbolic matrix (away from a discrete exceptional set of parameters), while it is \pregreat\ if for each parameter the image contains at least one hyperbolic matrix (away from a discrete exceptional set of parameters).
For operators, we say that certain sets are \ample\ if any chain of them can be perturbed in order to open a spectral gap at a specified energy (again away from an exceptional set of energies).
We prove the main abstract result in Section~\ref{sec:oneparam} and use it to verify that several examples of interest are \ample.
Section~\ref{sec:deducingThinness} then gives the proof that a generic limit-periodic word model over a \ample\ set enjoys Cantor spectrum of zero Hausdorff dimension.
We conclude with a brief subsection explaining the modifications needed to deal with the continuum setting.

\subsection*{Acknowledgements} 
J.F., M.N.G., and H.J.H.\ were supported in part by National Science Foundation Grant DMS 2213196. 
J.F.\ was supported in part by National Science Foundation Grant DMS 2513006
and Simons Foundation Grant MPS-TSM-00013720.
J.F.\ thanks the American Institute of Mathematics for hospitality at a recent SQuaRE program, where some of this work was done.

\section{Precise Statements of Results}\label{sec:results}

\subsection{Framework: Word Models}

Our first goal is to establish a coherent framework to discuss limit-periodic sequences in general spaces.
To that end,
suppose $(Y,d)$ is a complete metric space. 
For $k \in \bbN$, we will write $Y^k = \{\bm{y} = (y_j)_{j=1}^k\}$ and $Y^\bbZ = \{ \bm y = (y_j)_{j\in \bbZ}\}$, equipped with the induced uniform metrics (which by abusing notation we denote by the same letter)
\begin{equation}
\label{eq:unifMetricDef} d(\bm y, \bm z) := \sup_j d(y_j,z_j).
\end{equation}
We write $\per(\bbZ,Y)$ for the set of \emph{periodic} elements of $Y^\bbZ$, that is, those $\bm y$ for which there exists a \emph{period} $q \in \bbN := \{1,2,\ldots\}$ such that $y_{j+q} \equiv y_j$.
The set of limit-periodic $Y$-valued functions is then given by the closure of $\per(\bbZ,Y)$ with respect to the metric $d$ and is denoted by
\begin{equation}
\LP(\bbZ,Y): = \overline{\per(\bbZ,Y)}.
\end{equation}

Concatenations play an important role in the discussion below, so we will write a typical element of $Y^k$ as $\bm x = x_1 x_2 \cdots x_k$ and refer to it as a \emph{word} over $Y$. 
For $\bm x = x_1 \cdots x_k \in Y^k$ and $\bm y = y_1 \cdots y_\ell \in Y^\ell$, we write $\bm{x}  \concatenate\bm{y}$ for their concatenation $x_1 \cdots x_k y_1 \cdots y_\ell \in Y^{k+\ell}$, which we often just abbreviate as $\bm{xy}$; this makes
\begin{equation}
    Y^\star := \bigcup_{k=1}^\infty Y^k,
\end{equation}
the set of all words over $Y$, into a semigroup.
We also write   $\bm{x}^{\concatenate m}$ for the element of $Y^{mk}$ obtained by concatenating $m$ copies of $\bm x$:
   \[ \bm{x}^{\concatenate m}: = \underbrace{\bm{x} \concatenate \bm{x} \concatenate \cdots \concatenate\bm{x}}_{m \text{ copies}} ,\]
   and similarly $\bm{x}^{\concatenate\bbZ}$ denotes the $k$-periodic element of $\per(\bbZ,Y)$ that coincides with $\bm x$ on $\{1,2,\ldots,k\}$.
   
In the analysis that follows, we will need to compare elements of $Y^\star$ having different lengths.
   The relevant notion of approximation is supplied by asking how close their periodic extensions are to one another in the $\ell^\infty$ norm, that is, we will define
   \begin{equation}
       \label{eq:unifMetricDefUneven} 
       d(\bm x, \bm y) 
       := d(\bm x^{\concatenate\bbZ}, \bm y^{\concatenate\bbZ}), \quad \bm{x}, \bm{y} \in Y^\star.
       \end{equation}
   Let us point out that $d$ is no longer a metric on $Y^\star$ since, for example $d(\bm{x}^{\concatenate k}, \bm{x}^{\concatenate \ell})=0$ for any choice of $\bm x$, $k$, and $\ell$.
   Note however that $d(\bm x, \bm y)  =0$ if and only if $\bm x^{\concatenate \bbZ} = \bm y^{\concatenate  \bbZ}$.

Finally, observe that $y_0^{\concatenate\bbZ}$ is an isolated point of $\LP(\bbZ,Y)$ whenever $y_0$ is an isolated point of $Y$.
   Thus, to avoid trivialities in the discussion below, we will also assume that $Y$ is perfect, that is, without isolated points.

\subsection{One-Parameter Families of Matrices}
Denote by $\SL(2,\bbR)$ the collection of $2 \times 2$ matrices with entries in $\bbR$ and determinant $1$.
A matrix $A \in \SL(2,\bbR)$ is called:
\begin{itemize}
\item \emph{elliptic} if $|\tr(A)|<2$,
\item \emph{hyperbolic} if $|\tr(A)|>2$,
\item \emph{parabolic} if $|\tr(A)|=2$ and $A \neq \pm \idty$.
\end{itemize}

   As before, let $Y$ denote a perfect complete metric space. Assume that $A:Y \times \bbR \to \SL(2,\bbR)$ is jointly continuous as a function on $Y \times \bbR$. 
   We extend $A$ to the set of words and obtain a map $A:Y^\star \times \bbR \to \SL(2,\bbR)$ by insisting that
   \begin{equation} \label{eq:TMantiHomom}
   A(\bm{x} \concatenate \bm{y},\energy) = A(\bm y, \energy)A(\bm x, \energy), \quad \forall \bm x, \bm y \in Y^\star, \ \energy \in \bbR.
   \end{equation}
    For $k \in \bbN$, it will be useful to write $A_k$ for the induced map on $Y^k \times \bbR$:
   \begin{equation}
   \label{eq:AndefonXn} A_k(\bm x,\energy) = A(x_k,\energy) \cdots A(x_2,\energy) A(x_1,\energy), \quad \bm{x} = x_1\cdots x_k \in Y^k, \ \energy \in \bbR.
   \end{equation}

The following definition captures the primary input that is needed to open spectral gaps and run the spectral theoretic arguments later in the manuscript: away from a discrete exceptional set, any word can be perturbed to produce a new word for which the output of $A$ is hyperbolic.
\begin{definition}
Let $Y$ be a perfect complete metric space and $A : Y \times \bbR \to \SL(2,\bbR)$ be continuous as before.
We say that $A$ is  \emph{\great}\ with exceptional set $S \subseteq \bbR$ if $S$ is discrete and for all $\energy \in \bbR \setminus S$,   $\bm x \in Y^\star$, and $\varepsilon>0$ there exists $\bm y = \bm y(E,\bm{x},\varepsilon) \in Y^\star$ such that $d(\bm x, \bm y)<\varepsilon$ and  $A(\bm y, \energy)$ is hyperbolic.
\end{definition}

\begin{definition}
With $Y$ and $A$ as before, we say that $A$ is  \emph{\pregreat} with exceptional set $S \subseteq \bbR$ if $S$ is discrete and for all $\energy \in \bbR \setminus S$, there exists some $y \in Y$  such that $A(y,\energy)$ is hyperbolic.
\end{definition}

Here we emphasize the key advantage of the present idea over previous work: being \pregreat\ only needs to be checked ``at the first scale'' (i.e., on $Y \times \bbR$) whereas being \great\ must be checked ``at all scales'' (i.e., on all of $Y^\star \times \bbR$) and hence is in principle substantially more difficult to verify in practice.
Furthermore, being \pregreat\ is a ``global'' property: one simply needs to know that some element \textit{somewhere} in $Y$ produces a hyperbolic matrix, whereas \great ness contains ``local'' information: any word whatsoever can be perturbed slightly to open a spectral gap at a desired energy.
Our first result is that (under the assumption of \emph{analyticity}) every \pregreat\ map is \great.

\begin{theorem} \label{t:prerichImpliesRich}
Assume $X$ is a Banach space and that $A:X \times \bbR \to \SL(2,\bbR)$ is analytic. If $A$ is \pregreat, then it is \great.
\end{theorem}

Throughout this paper, we freely use the notion of analytic maps between Banach spaces. We point the reader to Federer for additional background and reference, especially \cite[Section~3.1.24]{Federer1969GMTBook}.

\subsection{Discrete Schr\"odinger Operators}

We will demonstrate the usefulness of the criterion described in Theorem~\ref{t:prerichImpliesRich} by showing that it can be applied to a wide variety of transfer matrix maps generated by discrete Schr\"odinger operators and that, when the criterion applies, one has Cantor spectrum of zero Hausdorff dimension for typical limit-periodic Schr\"odinger operators given by concatenating relevant strings.

Given a bounded, self-adjoint operator $H$ acting on a suitable Hilbert space, a fundamental quantity of interest is the \emph{spectrum}:
\begin{equation}
\spectrum H
= \set{\energy \in \bbC : H-\energy \idty \text{ is not invertible}},
\end{equation}
which is a compact subset of $\bbR$ under the given assumptions on $H$.
We also write
\[
\resolvent H = \bbC \setminus \spectrum H
\]
for the complementary set, called the \emph{resolvent set} of $H$.

\begin{definition}
    Define for $v,\energy \in \bbR$:
\begin{equation} \label{eq:TvEDefin}
 T(v,\energy) =\begin{bmatrix}
        \energy - v & -1 \\ 1 & 0
    \end{bmatrix}.
    \end{equation}

\end{definition}
The reader can verify that $u \in \bbC^\bbZ$ satisfies
\begin{equation}
u(n-1) + u(n+1) + V(n)u(n) = \energy u(n) \quad \forall n \in \bbZ
\end{equation}
if and only if
\begin{equation}
\begin{bmatrix} u(n+1) \\ u(n) \end{bmatrix}
= T(V(n),\energy) \begin{bmatrix} u(n) \\ u(n-1) \end{bmatrix}\quad \forall n \in \bbZ.
\end{equation}

For $\bm y = (y_j)_{j=1}^m \in (\bbR^k)^m$, $\bm y$ determines an element of $\bbR^{mk}$ in a natural manner by aggregation. We denote the aggretate of $\bm y$ by $\bm y^\aggregate$, that is, if $y_j = y_j(1)y_j(2) \cdots y_j(k)$, then
\begin{align}\nonumber
    \bm y^\aggregate
    & = ((y_1(1) \cdots y_1(k))(y_2(1) \cdots y_2(k)) \cdots (y_m(1) \cdots y_m(k)))^\aggregate \\
    & := y_1(1) \cdots y_1(k)y_2(1) \cdots y_2(k) \cdots y_m(1) \cdots y_m(k)
    \in \bbR^{mk}.
\end{align}
Similarly, given $\bm y = (y_j)_{j \in \bbZ} \in (\bbR^k)^\bbZ$, we define $\bm y^\aggregate \in \bbR^\bbZ$ by
\[ \bm y^\aggregate(jk+\ell)
= 
y_j(\ell),
\quad j \in \bbZ, \ 1 \le \ell \le k.
\]
We write the corresponding Schr\"odinger operator as $H_{\bm y} := \Delta + V_{\bm y} := \Delta + {\bm y}^\aggregate$. 
Given $\bm x \in (\bbR^k)^p$, we write $H_{\bm x}$ to mean $H_{\bm{x}^{\concatenate \bbZ}}$, where we recall that $\bm{x}^{\concatenate \bbZ } \in \per(\bbZ,\bbR^k)$ is obtained by repeating $\bm x$ $p$-periodically.
The following connection between spectra of periodic operators and transfer matrices is fundamental. 
See, e.g., \cite[Chapter~7]{DF2024ESO2}, \cite[Chapter~5]{Simon2011Szego}, or \cite[Chapter~7]{Teschl2000Jacobi} for proofs and further discussion.

\begin{theorem} \label{t:floquetmain}
If $q \in \bbN$ and $x \in \bbR^q$, then
\begin{equation}
\spectrum H_x 
= \set{\energy \in \bbR : \tr T(x,\energy) \in [-2,2]},
\end{equation}
where we recall that $T$ is defined on $\bbR \times \bbR$ by \eqref{eq:TvEDefin} and extended to $\bbR^q \times \bbR$ by concatenation as in \eqref{eq:TMantiHomom}.
Equivalently,
\begin{equation}
\bbR \setminus \spectrum H_x
= \set{\energy \in \bbR: T(x,\energy) \text{ is hyperbolic}}.
\end{equation}
\end{theorem}
In view of Theorem~\ref{t:floquetmain}, $\spectrum H_x$ consists of the $\energy \in \bbR$ for which the \emph{discriminant} 
\[
D(x,\energy):=\tr T(x,\energy)
\]
lies in $[-2,2]$. 
    One can show that the set of $\energy \in \bbR$ such that $|D(x,\energy)|<2$ consists of $q$ disjoint open intervals, which we denote by $(E_j^-, E_j^+)$, $j=1,2,\ldots,q$, ordered so that $E_1^- < E_1^+ \leq E_2^- < \cdots < E_q^+ $.
    Thus,
    \begin{equation} \label{eq:Ejpmdef}
        \spectrum H_x = \bigcup_{j=1}^q [E_j^-,E_j^+].
    \end{equation}
    We call the intervals $[E_j^-, E_j^+]$ with $j=1,2,\ldots, q$ the \emph{bands} of the spectrum of $H_x$. Note that consecutive bands need not necessarily be disjoint, but they must have disjoint interiors.

\begin{definition}
We say that a closed perfect set $Y \subseteq \bbR^k$ is \emph{\ample} with \emph{exceptional set} $S \subseteq \bbR$ if $S$ is discrete and the following holds true:
for every $\energy \in \bbR \setminus S$,  $\bm x \in Y^\star$, and $\varepsilon>0$, there exists $\bm y = \bm{y}(\energy,\bm{x},\varepsilon) \in Y^\star$ such that 
\begin{equation}
d(\bm{x}, \bm{y}) < \varepsilon
\quad \text{and}\quad
\energy \notin \spectrum(H_{\bm y}),
\end{equation}
where we recall $d$ is given by \eqref{eq:unifMetricDefUneven}.
\end{definition}

\begin{remark}
As an immediate consequence of the definitions and Theorem~\ref{t:floquetmain}, we have the following statement:
a set $Y \subseteq \bbR^k$ is \ample\ if and only if the map $Y \times \bbR \to \SL(2,\bbR)$ given by $(x,\energy) \mapsto T(x,\energy)$ is \great.
\end{remark}

Here, we give the main application of the notions.

\begin{theorem} \label{t:richToZHD}
    If $Y \subseteq \bbR^k$ is \ample, then $\spectrum(H_{\bm y})$ is a Cantor set of zero Hausdorff dimension for generic $\bm y \in \LP(\bbZ,Y)$.
\end{theorem}

For $Y \subseteq \bbR^k$ and $\lambda \in \bbR$, write $\lambda Y:=\{\lambda y: y \in Y\}$.
In the event that $\lambda Y$ is \ample\ for every $\lambda$ in some subset of $\bbR^\times:= \bbR \setminus\{0\}$, one can incorporate a coupling constant, as long as the exceptional set can be chosen uniformly over $\lambda$.

\begin{theorem} \label{t:richToZHDcoupling}
Suppose $S \subseteq \bbR$ is discrete, $Y \subseteq \bbR^k$ is perfect, and $Z \subseteq \bbR^\times$ is $\sigma$-compact.\footnote{Recall that a set is $\sigma$-compact if it can be written as the union of a countable collection of compact sets.}
If $\lambda Y$ is \ample\ with exceptional set $S$ for every $\lambda \in Z$, 
then $\spectrum(H_{\lambda \bm y})$ is a Cantor set of zero Hausdorff dimension for generic $\bm y \in \LP(\bbZ,Y)$ and all $\lambda \in Z$.
\end{theorem}

\begin{remark} Since $Y$ is a complete metric space, $\LP(\bbZ,Y)$ is a complete metric space with respect to the uniform metric, so genericity in this context means that the relevant phenomenon occurs for a dense $G_\delta$ set of realizations.
  \end{remark}

Let us give two interesting applications of the sufficient criterion in Theorem~\ref{t:prerichImpliesRich}
which can in turn be combined with Theorems~\ref{t:richToZHD} and \ref{t:richToZHDcoupling} to produce thin Cantor spectra and in particular prove Theorems~\ref{t:main:sieve} and \ref{t:main:repeat}.

\begin{theorem} \label{t:prerichSchroExampleSieving}
    For any $m \in \bbN$, any $\bm b \in \bbR^m$, and any $n \in \bbN$, the set $\{\bm{x} \concatenate\bm{b} : \bm{x} \in \bbR^n\}$ is \ample.
\end{theorem}

\begin{theorem} \label{t:prerichSchroExamplePolymer}
For any $n \in \bbN$, the set $\{x^{\concatenate n} : x \in \bbR\}\subseteq \bbR^n$ is \ample.
\end{theorem}

As mentioned above, once these theorems are proved, Theorems~\ref{t:main:sieve} and \ref{t:main:repeat} follow in short order.

\begin{proof}[Proof of Theorem~\ref{t:main:sieve}]
    Applying Theorem~\ref{t:prerichSchroExampleSieving} with $n=1$, $m=k-1$, and $\bm{b}={0}^{\concatenate(k-1)}$ shows that $Y_k:= \bbR \times \{0\}^{k-1}=\{x0^{\sharp(k-1)}: x \in \bbR\}$ is \ample. Since $\lambda Y_k = Y_k$ for all $\lambda \in \bbR^\times$, the conclusion follows by intersecting the dense $G_\delta$ sets obtained by applying Theorem~\ref{t:richToZHDcoupling} for each $k$.
\end{proof}

\begin{proof}[Proof of Theorem~\ref{t:main:repeat}]
    Similarly to the proof of Theorem~\ref{t:main:sieve}, this follows from Theorems~\ref{t:richToZHDcoupling} and \ref{t:prerichSchroExamplePolymer}.
\end{proof}

\subsection{Continuum Schr\"odinger Operators} \label{ssec:continuumops}
The framework of the current paper can be applied to other models.
To illustrate, we give a few applications to continuum Schr\"odinger operators on the real line.
Given $\varphi_1 \in L^2([0,a_1))$ and $\varphi_2 \in L^2([0,a_2))$, write
\begin{equation}
    (\varphi_1 \concatenate \varphi_2)(x)
    = \begin{cases}
        \varphi_1(x) & x \in [0,a_1) \\
        \varphi_2(x-a_1) & x \in [a_1,a_1+a_2)
    \end{cases}
\end{equation}
for their \emph{concatenation}.
More generally, for $\bm \varphi = (\varphi_j)_{j=1}^n$ with $\varphi_j \in L^2([0,a_j))$, write
\[ \bm \varphi^\aggregate = \varphi_1 \concatenate \cdots \concatenate \varphi_n \]
for the function in $L^2([0,\sum a_j))$ given by 
\[\bm\varphi^\aggregate(x)
=\varphi_j\left(x-\sum_{i=1}^{j-1}a_i\right) 
\text{ for } x \in \left[\sum_{i=1}^{j-1}a_i, \sum_{i=1}^{j} a_i\right).
\]
There are similar definitions of $\bm\varphi^\aggregate$ for $\bm\varphi = (\varphi_j)_{j=1}^\infty$ and $\bm \varphi = (\varphi_j)_{j \in \bbZ}$.

For a given $\varphi \in L^2([0,a))$ and $\energy \in \bbC$, the associated transfer matrix is given by $B(\varphi,\energy) = M(\varphi,\energy,a)$, where $M(\varphi,\energy,\cdot)$ is the solution to the initial value problem
\begin{equation} \label{eq:contSchro:IVP}
    \frac{\partial}{\partial x}M(\varphi,\energy,x)
    = \begin{bmatrix}
        0 & 1 \\
        \varphi(x)-\energy & 0
    \end{bmatrix}M(\varphi,\energy,x),
    \qquad M(\varphi,\energy,0) = \idty_{2 \times 2}.
\end{equation}

From the definitions, we note that
\begin{equation}
    B(\bm\varphi^\aggregate,\energy)
    = B(\varphi_n, \energy) \cdots B(\varphi_2,\energy)B(\varphi_1,\energy)
\end{equation}
and for the zero function in $L^2([0,a))$, one can explicitly solve the initial-value problem in \eqref{eq:contSchro:IVP} to obtain:
\begin{equation}
    B(0\cdot \chi_{_{[0,a)}}, \energy)
    = \begin{bmatrix}
        \cos(a\sqrt{E}) & \frac{\sin(a\sqrt{E})}{\sqrt{E}} \\
        -\sqrt{E} \sin (a\sqrt{E}) & \cos(a\sqrt{E})
    \end{bmatrix}.
\end{equation}

Similar to the discrete case, consider $Y$, a perfect subset of $L^2([0,a))$ for some $a>0$.
For $\bm y \in Y^\bbZ$, we write
\[ L_{\bm y} = -\partial_x^2 + \bm y^\aggregate \]
for the associated Schr\"odinger operator.
\begin{definition}
We say that a perfect set $Y \subseteq L^2([0,a))$ is \emph{\ample} with \emph{exceptional set} $S \subseteq \bbR$ if $S$ is discrete and for every $\energy \in \bbR \setminus S$, $\bm x \in Y^\star$, and $\varepsilon>0$, there exists $\bm y = \bm{y}(\energy,\bm{x},\varepsilon) \in Y^{\star}$ such that 
\begin{equation}
d(\bm{x}, \bm{y}) < \varepsilon
\quad \text{and}\quad
\energy \notin \spectrum(L_{\bm y}),
\end{equation}
where we recall $d$ is given by \eqref{eq:unifMetricDefUneven}.
\end{definition}

\begin{remark}
As an immediate consequence of the definitions and Floquet theory for continuum Schr\"odinger operators (see, e.g., \cite{Kuchment2016BAMS, Lukic2022book}) we have the following statement:
a set $Y \subseteq L^2([0,a))$ is \ample\ if and only if the map $Y \times \bbR \to \SL(2,\bbR)$ given by $(\varphi,\energy) \mapsto B(\varphi,\energy)$ is \great.
\end{remark}

\begin{theorem} \label{t:richToZHDcontinuum}
    If $Y \subseteq L^2([0,a))$ is \ample, then $\spectrum(L_{\bm y})$ is a Cantor set of zero Hausdorff dimension for generic $\bm y \in \LP(\bbZ,Y)$.
\end{theorem}

As before; in the event that $\lambda Y$ is \ample\ for a collection of $\lambda $'s with a uniform exceptional set, one can incorporate a coupling constant.

\begin{theorem} \label{t:richToZHDcouplingcontinuum}
    Suppose $S\subseteq \bbR$ is discrete, $Y \subseteq L^2([0,a))$ is perfect, and $Z \subseteq \bbR^\times$ is $\sigma$-compact. 
    If  $\lambda Y$ is \ample\ with exceptional set $S$ for every $\lambda  \in Z$, then $\spectrum(L_{\lambda \bm y})$ is a Cantor set of zero Hausdorff dimension for generic $\bm y \in \LP(\bbZ,Y)$ and all $\lambda \in Z$.
\end{theorem}

We conclude with two examples of \ample\ sets in $L^2$.

\begin{theorem} \label{t:continuum:sieve}
    Fix $a,b>0$, and choose $\psi \in L^2([0,b))$. 
    The set $\{\varphi \concatenate \psi : \varphi \in L^2([0,a))\}$ is \ample\ in $L^2([0,a+b))$.
\end{theorem}

\begin{theorem} \label{t:continuum:repeat}
    Let $a>0$ and $n \in \bbN$ be given, and consider $X = L^2([0,a))$.
    The set $\{\varphi^{\concatenate n} : \varphi \in L^2([0,a))\}$ is \ample\ in $L^2([0,na))$.
\end{theorem}
Let us conclude by emphasizing that the criteria espoused in this paper help to establish one of the main model-dependent parts in existing proofs of zero-measure spectrum and hence can be fruitfully applied in other scenarios.
To avoid many repetitive theorem statements, we do not write all of these out explicitly.
However, we point out that the other model-dependent aspects have been worked out for Jacobi matrices \cite{DamFilWan2023MANA} and CMV matrices \cite{FilOng2017JFA}, so one can easily obtain results in those settings as well.

\section{One-Parameter Families} \label{sec:oneparam}

The next main objective is to verify the simple sufficient criterion from Theorem~\ref{t:prerichImpliesRich}.
We begin with a few preparatory results.
Throughout the section, we assume that $X$ is a Banach space.

\begin{lemma} \label{lem:oneparam1}
    If $B_1 \in \SL(2,\bbR)$ is elliptic and $B_2 \in \SL(2,\bbR)$ is hyperbolic, then $B_1$ and $B_2$ do not commute.
\end{lemma}

\begin{proof}
    The eigenbasis of $B_1$ can be chosen to consist of a pair of linearly independent complex conjugates, whereas the eigenvectors of $B_2$ can be chosen to be real valued. 
    Thus, $B_1$ and $B_2$ may not be simultaneously diagonalized.
    \end{proof}

    \begin{lemma}\label{lem:oneparam2}
      Suppose $A:X \times \bbR \to \SL(2,\bbR)$ is analytic and \pregreat\ with exceptional set $S$. For every $\energy \in \bbR \setminus S$, $x \in X$, and $\varepsilon >0$, there exists $k \in \bbN$ and $\bm y \in X^k$ such that $d(x,\bm y)< \varepsilon$ and $A(\bm y, \energy)$ is hyperbolic.
    \end{lemma}
    
    \begin{proof} Let $\energy$, $x$, and $\varepsilon$ be given.  By definition, note that there exists $y_0$ such that $A(y_0,\energy)$ is hyperbolic.
    
\underline{Case 1:    $A(x,\energy)$ is hyperbolic.} In this case, the desired conclusion is satisfied with $k=1$ and $\bm y =x$.
    \medskip
    
\underline{Case 2:  $A(x,\energy)$ is elliptic.}
Then Lemma~\ref{lem:oneparam1} implies that $A(x,\energy)$ and $A(y_0,\energy)$ do not commute and therefore the map $X \times X \to \liesl(2,\bbR)$ given by
    \[ (y,y') \mapsto [A(y,\energy), A(y',\energy)] \]
    does not vanish identically and hence cannot vanish identically on any open subset of $X \times X$.
    Thus, we can choose $x_1,x_2 \in X$ within $\varepsilon$ of $x$ such that $A(x_1,\energy)$ and $A(x_2,\energy)$ are elliptic and do not commute.
    By \cite[Corollary~3.3]{EFGL2022JFA}, the semigroup generated by $\{A(x_1,\energy), A(x_2,\energy)\}$ contains a hyperbolic element. 
    Since each element of the semigroup generated by those matrices has the form $A(\bm x,\energy)$ for some $\bm x \in \{x_1,x_2\}^n$, the result follows in this case.
    \medskip
    
    \underline{Case 3: $|\tr A(x,\energy)| = \pm 2$.} Since $|\tr A(y_0,\energy)|>2$, it follows that $x \mapsto \tr A(x,\energy)$ is a nonconstant analytic function of $x$. 
    Thus, a small wiggle of $x$ puts us back into one of the previous two cases.
\end{proof}

\begin{lemma} \label{lem:oneparam3}
  Suppose $A:X \times \bbR$ is \pregreat\ with exceptional set $S$. 
For every $k \in \bbN$, the map $A_k:X^k \times \bbR \to \SL(2,\bbR)$ given by \eqref{eq:AndefonXn} is also \pregreat\ with the same exceptional set.
\end{lemma}

\begin{proof}
Since $A$ is \pregreat\ with exceptional set $S$, we may choose for each $\energy \in \bbR \setminus S$ an $x \in X$ such that $ A(x,\energy)$ is hyperbolic.
Applying \eqref{eq:TMantiHomom} to this $x$ and $\energy$, $ A(x^{\concatenate k},\energy) = A(x,\energy)^k$, which is also hyperbolic. 
Thus, the range of $A_k(\cdot, \energy)$ contains a hyperbolic element for every $E \in \bbR \setminus S$,
which concludes the proof.
\end{proof}

\begin{proof}[Proof of Theorem~\ref{t:prerichImpliesRich}]
Assume $A: X \times \bbR \to \SL(2,\bbR)$ is \pregreat\ with exceptional set $S$.
By Lemma~\ref{lem:oneparam3}, $A_k$ is also \pregreat\  with the same exceptional set $S$, so the desired conclusion then follows by applying Lemma~\ref{lem:oneparam2} to $A_k$.
\end{proof}

Let us now conclude with proofs that the various examples discussed earlier in the paper are themselves \great.

\begin{proof}[Proof of Theorem~\ref{t:prerichSchroExampleSieving}]
Let $m$, $\bm{b} \in \bbR^m$, and $n$ be given, put $X = \bbR^n$, and consider the map $A:X \times \bbR \to \SL(2,\bbR)$ given by
\begin{equation}
    A(\bm{x}, \energy)
    = T(\bm{x} \sharp \bm{b}, \energy),
\end{equation}
where $T$ is the Schr\"odinger transfer matrix map (given by \eqref{eq:TvEDefin} and extended to $X\times \bbR$ via \eqref{eq:TMantiHomom}).
In light of Theorem~\ref{t:prerichImpliesRich}, it suffices to show that $A$ is \pregreat.

Consider $M= (M_{ij}) \in \SL(2,\bbR)$, 
and note that
\[ M T(v,\energy) =
\begin{bmatrix} M_{11} & M_{12} \\ M_{21} & M_{22} \end{bmatrix}
\begin{bmatrix} \energy - v & -1 \\ 1 & 0 \end{bmatrix}
=
\begin{bmatrix}
M_{11} (\energy -v) +M_{12} & -M_{11} \\
M_{21}(\energy - v) + M_{22} & -M_{21}
\end{bmatrix}, \]
Thus, the trace of $MT(v,\energy)$ is a nonconstant affine function of $v$ whenever $M_{11} \neq 0$.

With that in mind, define 
\[ S = \{\energy \in \bbR :  (T(0^{\concatenate (n-1)}\bm b,\energy)_{11}  =0\},\] 
where we write $B_{ij}$ for the $ij$ entry of a matrix $B$. 
Notice that $(T(0^{\concatenate (n-1)}\bm b,\energy)_{11}$ is a monic polynomial in $\energy$ of positive degree (see, e.g., \cite[Proposition~2.2.5]{DF2022ESO}) and hence is nonconstant.
Since $S$ is the vanishing set of a non-constant polynomial, it is finite, hence  discrete.

By construction, we see that 
\[\bbR \ni v \mapsto \tr T(v 0^{\concatenate (n-1)} \bm{b}, \energy)
= \tr A(v 0^{\concatenate (n-1)}, \energy)\in \bbR\]
is a nonconstant affine function of $v \in \bbR$ whenever $\energy \notin S$. 
It follows that the indicated $A$ is \pregreat\ with exceptional set $S$; $A$ is therefore \ample\ by Theorem~\ref{t:prerichImpliesRich}, as desired.
\end{proof}

\begin{proof}[Proof of Theorem~\ref{t:prerichSchroExamplePolymer}]
As before,  it suffices to show that $(x,\energy) \mapsto T(x^{\concatenate n},\energy)$ is \pregreat.
Since the map $(x,\energy)\mapsto T(x,\energy)$ is \pregreat\ by inspection (with empty exceptional set), this follows  from the observation
\[ T(x^{\concatenate n},\energy) = (T(x,\energy))^n. \]
Concretely, if $T(x,\energy)$ is hyperbolic, then so is $T(x^{\concatenate n}, \energy)$.
\end{proof}

\begin{proof}[Proof of Theorem~\ref{t:continuum:sieve}]
    Fix $\psi$ and write 
    \[M(\energy) = B(\psi,\energy)
    =: \begin{bmatrix}
        \alpha(\energy) & \beta(\energy) \\ 
        \gamma(\energy) & \delta(\energy)
    \end{bmatrix}. \]
By general facts about solutions to the Schr\"odinger equation, $B(\varphi,\energy)$ is an analytic function of $(\varphi ,\energy) \in L^2([0,a)) \times \bbR$ \cite{PoschelTrubowitz1987}. 
Thus, by Theorem~\ref{t:prerichImpliesRich}, it suffices to check that $(\varphi,\energy)\mapsto B(\varphi \concatenate \psi,\energy)$ is \pregreat.
    For $\lambda \in \bbR$, note that
    \[ B(\lambda \cdot \chi_{_{[0,a)}},\energy) =
    B(0 \cdot \chi_{_{[0,a)}},\energy-\lambda )\]
    so
\[
        B(\lambda\cdot\chi_{_{[0,a)}} \concatenate \psi, \energy)
        =
        \begin{bmatrix}
        \alpha(\energy) & \beta(\energy) \\ 
        \gamma(\energy) & \delta(\energy)
    \end{bmatrix}
        \begin{bmatrix}
        \cos(a\sqrt{E  - \lambda}) & \frac{\sin(a\sqrt{E  - \lambda})}{\sqrt{E  - \lambda}} \\[2mm]
        -\sqrt{E  - \lambda} \sin (a\sqrt{E  - \lambda}) & \cos(a\sqrt{E  - \lambda})
    \end{bmatrix}        
\]
and consequently
\begin{align}\nonumber
\tr B(\lambda\cdot\chi_{_{[0,a)}} \concatenate \psi, \energy)
& = (\alpha(\energy)+\delta(\energy)) \cos(a\sqrt{\energy - \lambda}) \\
\nonumber
& \qquad+ \left( \frac{\gamma(\energy)}{\sqrt{\energy - \lambda}} - \beta(\energy)\sqrt{\energy - \lambda} \right) \sin(a\sqrt{\energy - \lambda}).
\end{align}
By considering the different possibilities for $\alpha,\ldots,\delta$, one sees that the absolute value of the right-hand side tends to $\infty$ as $\lambda \to \infty$ for any choice of $E$. 
Thus, there is always a choice of $\lambda$ making the resulting matrix hyperbolic, so the indicated $A$ is \pregreat\ with empty exceptional set (hence \ample\ by Theorem~\ref{t:prerichImpliesRich}).
\end{proof}

\begin{proof}[Proof of Theorem~\ref{t:continuum:repeat}]
    For $\lambda \in \bbR$, 
    \[
    B((\lambda\cdot \chi_{_{[0,a)}})^{\concatenate n},\energy)
    = \begin{bmatrix}
        \cos(an\sqrt{E  - \lambda}) & \frac{\sin(an\sqrt{E  - \lambda})}{\sqrt{E  - \lambda}} \\[2mm]
        -\sqrt{E  - \lambda} \sin (an\sqrt{E  - \lambda}) & \cos(an\sqrt{E  - \lambda})
    \end{bmatrix}.
    \]
    From this, we immediately see that the indicated $A$ is \pregreat\ with empty exceptional set.
\end{proof}

\section{Thin Spectra for Limit-Periodic Schr\"odinger Operators} \label{sec:deducingThinness}

\subsection{Thin spectra for periodic discrete operators}
The main fact we need to prove Theorem~\ref{t:richToZHD} is a suitable approximation result that perturbs a given periodic sequence to produce one with exponentially thin spectrum.

\begin{theorem} \label{t:thinspecA}
    Assume $Y \subseteq \bbR^k$ is \ample\ with exceptional set $S$. For any $\bm{x} \in \per(\bbZ,Y)$ of minimal period $p$, any $\varepsilon>0$, and any compact $K \subseteq \bbR \setminus S$, there exist constants $N_0=N_0(\bm{x}, K, \varepsilon)$ and $c_0 = c_0(\bm{x}, K, \varepsilon)>0$ such that for any $N \geq N_0$, there exists $\bm{y} \in \per(\bbZ,Y)$ of period $Np$ such that 
    \begin{equation}
        d(\bm x, \bm y) < \varepsilon,
        \quad 
        \Leb(K \cap \spectrum(H_{\bm y})) \leq e^{-c_0 N}.
    \end{equation}
\end{theorem}

Naturally, to prove Theorem~\ref{t:richToZHDcoupling}, we will need a version that incorporates a coupling constant.

\begin{theorem} \label{t:thinspecB}
    Assume $S \subseteq \bbR$ is discrete, $Y \subseteq \bbR^k$ is perfect,  $Z \subseteq \bbR^\times $ is compact, and that $ \lambda Y$ is \ample\ with exceptional set $S$ for every $\lambda  \in Z$. 
    For any $\bm{x} \in \per(\bbZ,Y)$ of minimal period $p$, any $\varepsilon>0$, and any compact $K \subseteq \bbR \setminus S$, there exist constants $N_0=N_0(\bm{x}, K, \varepsilon,Z)$ and $c_0 = c_0(\bm{x}, K, \varepsilon,Z)>0$ such that for any $N \geq N_0$, there exists $\bm{y} \in \per(\bbZ,Y)$ of period $Np$ such that 
    \begin{equation}
        d(\bm x, \bm y) < \varepsilon,
        \quad 
        \Leb(K \cap \spectrum (H_{\lambda \bm y})) \leq e^{-c_0 N}
    \end{equation}
    for all $\lambda \in Z$.
\end{theorem}

\begin{proof}[Proof of Theorem~\ref{t:thinspecA}]
    This follows from Theorem~\ref{t:thinspecB} by taking $Z = \{1\}$.
\end{proof}

We now turn to proving Theorem~\ref{t:thinspecB}. This follows the contours of the proof of \cite[Lemma~3.1]{DamFilLuk2017JST}.
Indeed the notion of a set being \ample\ distills exactly what is needed to run earlier proofs, and furthermore, without additional information, would be somewhat difficult to verify in most settings.
Thus, we reiterate that the main contribution of the current paper is the simple sufficient criterion encoded in Theorem~\ref{t:prerichImpliesRich}.

In order to formulate the proofs, we need to discuss the spectral theory of periodic operators and some derived quantities in more detail; for proofs and detailed discussions, see \cite[Chapter~7]{DF2024ESO2}, \cite[Chapter~5]{Simon2011Szego}, or \cite[Chapter~7]{Teschl2000Jacobi}.

\begin{definition}
    Let $q \in \bbN$ and $x \in \bbR^q$ be given. The \emph{Lyapunov exponent} associated to the $q$-periodic operator $H_x$ is the function
    \begin{equation}
    \lyap(\energy)=
        \lyap(x,\energy)
        = \frac{1}{q} \log \spr T(x,\energy)
, 
    \quad \energy \in \bbC,
    \end{equation}
    where $\spr A$ denotes the spectral radius of the matrix $A$.
    By the spectral radius formula, this is equivalent to
    \begin{equation}
        \lyap(x, \energy)         = \lim_{n \to \infty} \frac{1}{nq} \log\| T(x^{\concatenate n},\energy) \| .
    \end{equation}
    In particular, Theorem~\ref{t:floquetmain} implies that
    \begin{equation}
        \spectrum(H_x) = \{ \energy \in \bbR : \lyap(x,\energy) = 0 \}.
    \end{equation}
    We extend this to blocks in the natural manner: given $k \in \bbN$ and $\bm y \in (\bbR^k)^q$, we write
    \begin{equation} \label{eq:LYAPDEF}
        \lyap (\bm y, \energy) = \frac{1}{kq} \log \spr T(\bm y,\energy).
    \end{equation}
    Notice that in terms of the natural aggregation map from before one has $\lyap(\bm  y, \energy) = \lyap(\bm y^\aggregate,\energy)$.
\end{definition}

\begin{definition}
    Let $q \in \bbN$ and $x \in \bbR^q$ be given. The \emph{integrated density of states} associated with the periodic operator $H_x$ is the function
    \begin{equation}
    \ids(\energy)=
        \ids(x,\energy)
        = \frac{1}{q} \sum_{j=0}^{q-1} \langle \delta_j, \chi_{_{(-\infty, \energy]}}(H_x) \delta_j \rangle.
    \end{equation}
\end{definition}

Crucial for us is the following property: $\ids(x,\cdot)$ is a continuous nondecreasing function such that
\begin{equation}\label{eq:DOSbandmeas}
    \ids(x,E_j^+) - \ids(x,E_j^-) = \frac{1}{q}
\end{equation}
for every band $[E_j^-,E_j^+]$ of the spectrum.
See \cite[Chapter~7]{DF2024ESO2} for more details about the IDS of periodic operators.

We also need the following device using the M\"obius transformation associated with a $2\times 2$ matrix.
Given 
\[A = \begin{bmatrix}
    a & b \\ c & d
\end{bmatrix},\]
we write
\[Az = \frac{az+b}{cz+d}, \quad z \in \bbC.\]
For any $( x,\energy)$ such that $T( x,\energy)$ has trace in $(-2,2)$, there is a unique $z_+ = z_+( x, \energy) \in \bbC_+ := \{z \in \bbC : \Im z >0\}$ such that $T( x, \energy) z_+ = z_+$. Defining 
\[ M( x, \energy) = \frac{1}{[\Im z_+( x, \energy)]^{1/2}} \begin{bmatrix}
    1 & - \Re z_+( x, \energy) \\
    0 & \Im z_+( x, \energy)
\end{bmatrix},
\]
we note that $M(x,\energy) \in \SL(2,\bbR)$.
It is well-known and not hard to show that $M$ conjugates $T$ to a rotation, that is, 
\begin{equation} \label{eq:MTMtoSO2}
    M(x,\energy) T(x,\energy)M(x,\energy)^{-1} \in \SO(2,\bbR) \text{ whenever } \tr T(x,\energy) \in (-2,2)
\end{equation}
and that any other $M'$ satisfying \eqref{eq:MTMtoSO2} must be of the form $OM$ for some $O \in \SO(2,\bbR)$.

The following fact connects these quantities: given $q \in \bbN$ and $x \in \bbR^q$, then for  any $\energy$ such that $\tr T(x,\energy) \in (-2,2)$, one has
    \begin{equation} \label{eq:dkdEviaconjugacy}
        \frac{\partial}{\partial E} \ids(x,\energy) = \frac{1}{4\pi q} \sum_{j=0}^{q-1} \|M(\cyclic^j x,\energy)\|^2_2
    \end{equation}
    where  $\cyclic : \bbR^q \to \bbR^q$ denote the cyclic permutation $x_1x_2\cdots x_{q-1} x_q \mapsto x_2x_3\cdots x_{q-1}x_qx_1$ and $\|A\|_2 = (\tr(A^*A))^{1/2}$ denotes the Hilbert--Schmidt norm of the matrix $A$.
    See \cite[Theorem~7.3.10]{DF2024ESO2} for a proof; see also \cite{Avila2009CMP, AviDam2008Invent, DeiSim1983CMP} for related discussions about the absolutely continuous spectrum.

\begin{proof}[Proof of Theorem~\ref{t:thinspecB}]
Suppose $S \subseteq \bbR$,  $Y \subseteq \bbR^k$, and $Z \subseteq \bbR^\times$ are as in the statement of the theorem.
Let $\bm x \in \per(\bbZ, Y)$, $K \subseteq \bbR \setminus S$ compact, and $\varepsilon >0$  be given, let $p\in \bbN$ denote the minimal period of $\bm {x}$, and write $\bm{x} = \bm{a}^{\concatenate \bbZ}$ for a suitable $\bm{a} = a_1 \cdots a_p \in Y^p$.
Notice that the induced potential $V_{\bm x} := {\bm x}^\aggregate$ is then $kp$-periodic as a function $\bbZ \to \bbR$.
     Taking $\varepsilon$ sufficiently small, we can ensure that the (minimal) period of any periodic $\bm y \in \per(\bbZ,Y)$ with $d(\bm x, \bm y) < \varepsilon$ must be a multiple of $p$. \begin{comment}
    To see this, assume $\varepsilon$ is small enough that 
    \[ \varepsilon < \frac{1}{2} d(a_j,a_k) \quad \text{whenever } a_j \neq a_k. \]
Suppose now that $d(\bm x, \bm y)< \varepsilon$ and $q$ is the minimal period of $\bm y$.
Denoting $d = \gcd(p,q)$, choose integers $s,t$ with $sp+tq=d$, and observe that
    \begin{align*}
        d(x_{n+d},x_n)
        & = d(x_{n+sp+tq}, x_{n}) \\
        & = d(x_{n+tq}, x_{n}) \\
        & \leq d(x_{n+tq} , y_{n+tq}) + d(y_{n},x_{n}) \\
        & < 2\varepsilon.
    \end{align*} 
    By construction, it follows that $x_{n+d} \equiv x_n$, that is, $\bm x$ is $d$-periodic. 
    Thus, the greatest common divisor of $p$ and $q$ is $p$, from which it follows that $q$ itself is a multiple of $p$.\end{comment}

We now wish to perturb (uniformly in the coupling constant) to open spectral gaps. 

\begin{claim} There exists $t \in \bbN$ and a finite set $\scrF \subseteq Y^{tp}$ such that
\begin{equation}
K \subseteq \bigcup_{\bm b \in \scrF} \resolvent(H_{\lambda \bm b})
\end{equation}
for all $\lambda \in Z$.
\end{claim}

\begin{claimproof}
For each $(\energy,\lambda) \in K \times Z $, the assumption that $\lambda Y$ is \ample\ (with exceptional set $S \subseteq \bbR \setminus K$) allows us to choose $k = k(E,\lambda) \in \bbN$ and  $\bm b= \bm b^{\energy,\lambda} \in Y^{kp}$ such that 
\begin{equation}
d(\bm a, \bm b^{\energy, \lambda}) < \varepsilon, \quad
\energy \notin \spectrum H_{\lambda \bm b^{\energy, \lambda}}.
\end{equation}
For each fixed $\lambda \in Z $, the collection
\[ \{\resolvent(H_{\lambda \bm b^{\energy, \lambda}}) : \energy \in K \}\]
forms an open cover of $K$ and hence one can extract $s(\lambda) \in \bbN$ and $E_1(\lambda) , E_2(\lambda), \ldots , E_{s(\lambda)}(\lambda)$ such that
\begin{equation}
K \subseteq
\bigcup_{j=1}^{s(\lambda)} \resolvent(H_{\lambda \bm b^{\energy_j(\lambda), \lambda}}).
\end{equation}
Moreover, by continuity of the spectrum, there is an open interval $I(\lambda) \ni \lambda$ such that
\begin{equation}
K \subseteq
\bigcup_{j=1}^{s(\lambda)} \resolvent(H_{\widetilde\lambda \bm b^{\energy_j(\lambda), \lambda}})
\end{equation}
for all $\widetilde\lambda \in I(\lambda)$.
By compactness, there then exist $\lambda_1,\ldots,\lambda_r \in Z $ such that $\{I(\lambda_i)\}_{i=1}^r\}$ covers $Z $.
The collection of $\bm c^{i,j} = \bm{b}^{E_j(\lambda_i),\lambda_i}$ with $1 \le i \le r$ and $1 \le j \le s(\lambda_i)$ is essentially the desired finite set, modulo normalization so that they all have the same period. 
More precisely, take 
\[
t = \lcm\{ k(E_j(\lambda_i),\lambda_i) : 1 \le i \le r, \ 1 \le j \le s(\lambda_i)\}, \quad 
p_{i,j} = t/k(E_j(\lambda_i), \lambda_i).
\]
The claim now holds with $\scrF = \{( \bm c^{i,j})^{\concatenate p_{i,j}} : 1 \le i \le r, \ 1 \le j \le s(\lambda_i)\}$.
\end{claimproof}
\bigskip

Write $m = \#\scrF$ and reindex the elements as $\scrF = \{\bm c_1, \ldots,\bm c_m\}$. 
Fixing $N \in \bbN$ large, choose $u \in \bbN$ maximal with $mtu \leq N$, and define
\[ \bm y = \left[ (\bm c_1)^{\concatenate  u} (\bm c_2)^{\concatenate  u}  \cdots (\bm c_m)^{\concatenate  u}  (\bm a)^{\concatenate  (N-mtu)} \right]
\in Y^{Np}. \]
By continuity of the Lyapunov exponent and compactness, we have
\begin{equation}
    L_{\rm min} :=  \inf\left\{ \max_{1\le j \le m} \lyap (\lambda \bm c_j, \energy) : (\energy,\lambda) \in K \times Z  \right\} > 0.
\end{equation}

Fix $\lambda \in Z $, and consider $\energy \in K$ such that $\tr T(\lambda \bm y, \energy) \in (-2,2)$; notice that $\energy$ must belong to $\spectrum H_{\lambda \bm y}$.
However, by construction of $\scrF$, there must be some $1 \le j \le m$ for which $\energy \notin \spectrum H_{\lambda \bm c_j}$, and thus
\begin{align}\nonumber
    \|T(\lambda \bm c_j^{\concatenate u},\energy)\|
    =\|[T(\lambda \bm c_j,\energy)]^u\| 
     \geq \spr( [T(\lambda \bm c_j,\energy)]^u)
    & = (\spr T(\lambda \bm c_j,\energy))^u \\
    \nonumber
    & = \exp(kptu \  \lyap(\lambda \bm c_j, \energy)) \\
    \label{eq:TlambdacjexpLB}
    & \geq \exp(kptu L_{\rm min}),
\end{align}
where the second line follows from \eqref{eq:LYAPDEF}.

\begin{claim}
For 
\[ X = M(\lambda  \cyclic^{(j-1)kptu} \bm y^\aggregate,\energy)  \text{ or } 
X = M(\lambda \cyclic^{jkptu} \bm y^\aggregate, \energy) T(\lambda \bm c_j^{\concatenate u}, \energy) \]
one has\footnote{Note that we need to pass from $\bm y \in Y^{Np}$ to the aggregated vector $\bm y^\aggregate \in \bbR^{Npk}$ to correctly compute cyclic permutations and other quantities here.}
 $X T(\lambda \cyclic^{(j-1)kptu} \bm y^\aggregate,\energy) X^{-1} \in \SO(2,\bbR)$.
 \end{claim}
 
 \begin{claimproof}
For $X = M(\lambda  \cyclic^{(j-1)kptu} \bm y^\aggregate,\energy) $, this follows directly from \eqref{eq:MTMtoSO2}.
For the other case, we note that
\begin{align*}
    &  T(\lambda \bm c_j^{\concatenate u}, \energy) 
    T(\lambda \cyclic^{(j-1)kptu} \bm y^\aggregate,\energy)
    T(\lambda \bm c_j^{\concatenate u}, \energy)^{-1}  
%    & \qquad = T(\lambda \bm c_j^{\concatenate u}, \energy) 
%    T(\lambda  (\bm c_{j}^{\concatenate  u}  \cdots \bm c_m^{\concatenate  u} \bm a^{\concatenate  (N-mtu)} \bm c_1^{\concatenate u}  \cdots \bm c_{j-1}^{\concatenate u}),\energy) 
%    T(\lambda \bm c_j^{\concatenate u},\energy)^{-1}  \\
%& \qquad = T(\lambda  (\bm c_{j+1}^{\concatenate  u}  \cdots \bm c_m^{\concatenate  u}  \bm a^{\concatenate  (N-mtu)} \bm c_1^{\concatenate u}  \cdots \bm c_{j}^{\concatenate u}),\energy)\\
   = T(\lambda \cyclic^{jkptu} \bm y^\aggregate,\energy).
\end{align*}
Thus, for $X = M(\lambda \cyclic^{jkptu} \bm y, \energy) T(\lambda \bm c_j^{\concatenate u}, \bm y)$, we have
\begin{align*}
   & \, XT(\lambda \cyclic^{(j-1)kptu} \bm y,\energy)X^{-1}\\
   & \qquad
   = 
   M(\lambda \cyclic^{jkptu} \bm y, \energy)
   T(\lambda \cyclic^{jkptu} \bm y,\energy)
   M(\lambda \cyclic^{jkptu} \bm y, \energy)^{-1},
\end{align*}
which again belongs to $\SO(2,\bbR)$ by definition. 
\end{claimproof}
\medskip

Consequently,
\begin{equation} \label{eq:MandShiftedMreln}
M(\lambda  \cyclic^{(j-1)kptu} \bm y,\energy) 
=O M(\lambda \cyclic^{jkptu} \bm y, \energy) T(\lambda \bm c_j^{\concatenate u}, \bm y) 
\end{equation}
for some $O \in \SO(2,\bbR)$ by the uniqueness of conjugacies (see, e.g., \cite[Proposition~1.13.7]{DF2022ESO}).
Combining \eqref{eq:MandShiftedMreln} with \eqref{eq:TlambdacjexpLB}, we deduce
\begin{equation} \label{eq:mainTMLB}
    \max\big\{
\|M(\lambda  \cyclic^{(j-1)kptu} \bm y^\aggregate,\energy) \|,
\|M(\lambda \cyclic^{jkptu} \bm y^\aggregate, \energy) \| 
\big\}
\geq \exp(\tfrac{1}{2}kptuL_{\min}).
\end{equation}
Notice that this lower bound is uniform over the range of $\energy$ under consideration.

Putting together \eqref{eq:mainTMLB} with \eqref{eq:dkdEviaconjugacy}, we have established the following: 
for all $\energy \in \spectrum H_{\lambda \bm y} \cap K$ for which $\tr T(\lambda \bm y, \energy) \in (-2,2)$, one has
\begin{equation} \label{eq:DkDEbound}
    \frac{\partial}{\partial \energy} \ids(\lambda \bm y, \energy) 
    \geq \frac{1}{4 \pi Npk} \exp(kptuL_{\min})
    \gtrsim \exp(c_0  N).
\end{equation}
Thus, for each band $B$ of $\spectrum H_{\lambda \bm y}$, \eqref{eq:DkDEbound} combined with \eqref{eq:DOSbandmeas} yields
\begin{equation}
    \frac{1}{Np} = \int_B \frac{\partial}{\partial E} \ids(\lambda\bm{y},\energy) \, \mathrm{d}E
    \geq \int_{B \cap K} \frac{\partial}{\partial E} \ids(\lambda\bm{y},\energy) \, \mathrm{d}E
    \gtrsim \Leb(B \cap K) e^{c_0 N},
\end{equation}
yielding
$\Leb (B \cap K) \lesssim e^{-c_0  N}$.
Since the spectrum associated with $\lambda \bm{y}$ consists of $Np$ bands, the result follows after adjusting the constants.
\end{proof}

\subsection{Thin spectra for limit-periodic discrete operators}

With the key preliminary results (Theorem~\ref{t:thinspecA} and \ref{t:thinspecB}) in hand, we can now prove Theorems~\ref{t:richToZHD} and \ref{t:richToZHDcoupling}.
For the reader's convenience, we briefly recall relevant definitions related to Hausdorff measure and Hausdorff dimension; 
see Falconer \cite{Falconer1990fractals} or Matilla \cite{Mattila1995book} for further discussion.
Given a set $S \subseteq \bbR$, a \emph{$\delta$-cover} $\mathcal{I} = \{I_j\}$ of $S$ is a collection of intervals of length at most $\delta$ whose union contains $S$. 
The $\alpha$-dimensional \emph{Hausdorff measure} is then given by
\begin{equation}
    h^\alpha(S) = \lim_{\delta \downarrow 0} \inf\left\{ \sum_{I \in \mathcal I} |I|^\alpha : \mathcal I \text{ is a } \delta\text{- cover of } S \right\}.
\end{equation}
For each nonempty $S \subseteq \bbR$, there exists $\alpha_0 \in [0,1]$ with the property that $h^\alpha(S)$ is infinite if $0 \leq \alpha < \alpha_0$ and vanishes if $\alpha_0 < \alpha \leq 1$.
This unique value is known as the \emph{Hausdorff dimension} of $S$ and is denoted by $\alpha_0 = \dim_\Hausdorff(S)$.
The \emph{lower box counting dimension} of a bounded set $S$ is given by
\begin{equation}\label{BOXCOUNTDEF}
    \dim_{\rm B}^-(S)
    =\liminf_{\varepsilon\to 0} \frac{\log N(S,\varepsilon)}{\log (1/\varepsilon)}
\end{equation}
where $N(S,\varepsilon)$ is the minimal number of intervals of length $\varepsilon$ needed to cover $S$.

\begin{proof}[Proof of Theorem~\ref{t:richToZHDcoupling}]
    Since discrete Schr\"odinger operators with bounded real-valued potentials are bounded self-adjoint operators, their spectra are always compact subsets of $\bbR$.
    \begin{claim} \label{cl:perfect}
       For any $\bm x \in \LP(\bbZ,Y)$, the spectrum of $H_{\bm x}$ does not contain isolated points.
    \end{claim}
    \begin{claimproof}
        This follows from general principles (compare \cite{Pastur1980CMP}).
        Concretely, note that $\bm x^{\aggregate} \in \ell^\infty(\bbZ,\bbR)$ is almost-periodic, so we can consider its hull $\Omega_{
        \bm x} := \operatorname{hull}(\bm x^\aggregate) \subseteq \ell^\infty(\bbZ)$ (which is the closure in the uniform topology of the orbit of $\bm x^\aggregate$ under the action of the shift $[S\omega](n) = \omega(n+1)$).
        Then $\Omega_{\bm x}$ is compact, and $(\Omega_{\bm x},S)$ is strictly ergodic: since the hull is a compact abelian group and the shift is a minimal translation thereupon, this follows from standard results, e.g., \cite[Theorem~C.2.10]{DF2024ESO2}.
        Thus, the spectrum of $H_{\bm x}$ coincides with the spectrum of $H_\omega$ for any $\omega \in \Omega_{\bm x}$ by \cite[Theorem~4.9.1]{DF2022ESO}, which has no isolated points by \cite[Theorem~4.2.4]{DF2022ESO}.
    \end{claimproof}
\medskip

    \begin{claim} \label{cl:gdelta}
    The set $Q$ of $\bm y\in \LP(\bbZ,Y)$ such that $\dim_\Hausdorff \spectrum H_{\lambda \bm y}=0$ for all $\lambda \in Z$ is a $G_\delta$ set.
\end{claim}

\begin{claimproof}
For $\delta,\varepsilon,\alpha > 0$ and $Z' \subseteq Z$ compact, let us say that $\bm{y} \in \LP(\bbZ,Y)$ is $(\delta,\varepsilon,\alpha,Z')$-\emph{covered} if for each $\lambda \in Z'$ there exist open intervals $I_1,\ldots,I_\ell \subseteq \bbR$ of length at most $\delta$ whose union contains $\spectrum H_{\lambda \bm y}$ such that
\begin{equation}
\sum_{j=1}^\ell |I_j|^\alpha < \varepsilon.
\end{equation}
Denote
\begin{equation}
U_{\delta,\varepsilon,\alpha,Z'} 
= \set{\bm{y} \in \LP(\bbZ,Y) : \bm{y} \text{ is } (\delta,\varepsilon,\alpha,Z')\text{-covered}}.
\end{equation}
By continuity of the spectrum and compactness of $Z'$, $U_{\delta,\varepsilon,\alpha,Z'}$ is an open set for all $\varepsilon,\delta, \alpha > 0$. 
Since
\begin{equation}
Q
= \bigcap_{\delta>0}  \bigcap_{\varepsilon>0} \bigcap_{\alpha>0} \bigcap_{Z'} U_{\delta, \varepsilon, \alpha, Z'},
\end{equation}
 we can choose countable sequences $\delta_n,\varepsilon_n,\alpha_n\downarrow 0$ and $Z_n' \uparrow Z$ to see that $Q$ is a $G_\delta$ set, as promised.
\end{claimproof}
\medskip
    
    Due to Claim~\ref{cl:perfect}, the spectrum of $H_{\lambda \bm y}$ never has isolated points. 
    Thus, to show that $\spectrum H_{\lambda  \bm y}$ is a Cantor set of zero Hausdorff dimension for all $\lambda  \in Z$ it suffices to show $\dim_\Hausdorff \spectrum H_{\lambda \bm y}=0$ for all $\lambda \in Z$.
    The set of such $\bm y$ is a $G_\delta$ set in $\LP(\bbZ,Y)$ by Claim~\ref{cl:gdelta}.
Putting this all together, since $\per(\bbZ, Y)$ is dense in $\LP(\bbZ,Y)$, to conclude the proof it suffices to show that for any $\bm x_0 \in \per(\bbZ,Y)$ and any $\varepsilon_0 > 0$, there exists $\bm y \in \LP(\bbZ,Y)$ within $\varepsilon_0$ of $\bm x_0$ such that 
\begin{equation}
    \dim_\Hausdorff \spectrum H_{\lambda \bm y} = 0 \text{ for all } \lambda \in Z.
\end{equation}

To that end, fix $\bm x_0 \in \per(\bbZ,Y)$ of period $p_0$ and $\varepsilon_0\in (0,1)$, and recall that $S$ denotes the common exceptional set of  $\lambda Y$ for every $\lambda \in Z$.
Fix a sequence $Z_\ell \subseteq Z$ of compact sets increasing to $Z$.
Also fix  $\eta_0>0$ small enough that
\[
F_0= [-\eta_0^{-1},\eta^{-1}_0]\setminus B_{\eta_0}(S)
\]
overlaps with $\interior \spectrum H_{\lambda \bm{x}_0}$ for all $\lambda \in Z_0$.

Now use Theorem~\ref{t:thinspecB} inductively to choose decreasing sequences $\varepsilon_\ell ,\eta_\ell \downarrow 0$ and
$\bm x_\ell \in \per(\bbZ, Y)$ of period $p_\ell$ for $\ell \in \bbN$,
such that for $F_\ell= [-\eta_\ell^{-1},\eta^{-1}_\ell]\setminus B_{\eta_\ell}(S)$, one has $\interior F_\ell \cap \spectrum H_{\lambda \bm{x}_\ell} \neq \emptyset$ for all $\ell$ and $\lambda \in Z_\ell$,
\begin{align}
\label{eq:thinspec:espcond1}
&\varepsilon_\ell < \tfrac{1}{2} \min \left\{\varepsilon_{\ell-1},  \tfrac{1}{4} \min_{\lambda \in Z_{\ell-1}} \Leb\left(F_{\ell-1} \cap \spectrum(H_{\lambda \bm{x}_{\ell-1}}) \right) \right\} & \forall  \ell \in \bbN, \\
\label{eq:thinspec:xkcond1}
&d(\bm{x}_\ell, \bm{x}_{\ell-1})  < \varepsilon_\ell & \forall \ell \in \bbN, \\
\label{eq:thinspec:speckcond}
&\Leb (F_{\ell} \cap \spectrum H_{\lambda \bm {x}_\ell}) < e^{-p_\ell^{1/2}} & \forall \ell \in \bbN \text{ and } \lambda \in Z_{\ell}.
\end{align}
Combining \eqref{eq:thinspec:espcond1} and \eqref{eq:thinspec:xkcond1}, we see that there is $\bm{x}_\infty$ such that $\bm{x}_\ell \to \bm{x}_\infty$ uniformly and moreover
\begin{equation}
    d(\bm{x}_\ell, \bm{x}_\infty)
    < \sum_{j=\ell+1}^\infty \varepsilon_j < \varepsilon_\ell.
\end{equation}
In particular, $\bm{x}_\infty \in \LP(\bbZ,Y)$ and $d(\bm{x}_0,\bm{x}_\infty)<\varepsilon_0$.

For each $\lambda$ and $0 \le \ell \le \infty$, denote $\Sigma_{\lambda,\ell} = \spectrum(H_{\lambda \bm{x}_\ell})$.
To conclude the proof, it suffices to show that
\begin{equation}
    h^\alpha(\Sigma_{\lambda,\infty})=0 \quad \text{ for all } \lambda \in Z \text{ and } 0<\alpha <1.
\end{equation}
To that end, fix $\lambda \in Z$ and $\ell \in \bbN$ large enough that $\lambda \in Z_\ell$.
We will show for each $m \geq \ell$ that $\Sigma_{\lambda,\infty} \cap F_\ell$ can be covered by relatively few intervals of length $2 \exp(-p_m^{1/2})$.
We emphasize here that the set $\Sigma_{\lambda,\infty}\cap F_\ell$ is fixed at this stage; what varies here is the length of intervals used to cover this set.
To that end, combining \eqref{eq:thinspec:espcond1}, \eqref{eq:thinspec:xkcond1}, and \eqref{eq:thinspec:speckcond}, we have the following for any $m \geq \ell$:
\begin{equation}
    d(\bm{x}_m, \bm{x}_\infty) < \sum_{\ell=m+1}^\infty \varepsilon_\ell< 2 \varepsilon_{m+1} < \tfrac{1}{4} \exp(- p_m^{1/2}).
\end{equation}
Consequently, we see that for each $m \geq \ell$, $\Sigma_{\lambda,\infty} \cap F_\ell$ can be covered by a collection $\{I_j\}$ of no more than $p_m+C(\ell)$ intervals of length at most $2 \exp(-p_m^{1/2})$, where $C(\ell)$ is some constant only depending on $\ell$ (which is fixed for now).\footnote{Comparing this with \eqref{BOXCOUNTDEF} we see that this statement implies that $\Sigma_{\lambda,\infty} \cap F_\ell$ has lower box counting dimension zero.}
This leads to
\begin{equation}
    \sum_j |I_j|^\alpha
    \lesssim  p_m e^{-\alpha p_m^{1/2}}.
\end{equation}
Since this tends to zero as $m \to \infty$, $h^\alpha(\Sigma_{\lambda,\infty} \cap F_\ell)=0$.
It follows that $\dim_\Hausdorff \Sigma_{\lambda, \infty} =0$, as desired.
\end{proof}

\begin{proof}[Proof of Theorem~\ref{t:richToZHD}]
This follows from Theorem~\ref{t:richToZHDcoupling} with $Z=\{1\}$.
\end{proof}

\subsection{Continuum Schr\"odinger operators}

Let us conclude by going over the modifications necessary to generalize the framework above to continuum Schr\"odinger operators.
As before, the first step is to perturb a given periodic operator to produce thin spectrum:

\begin{theorem} \label{t:thinspecAcontinuum}
    Assume $Y \subseteq L^2([0,a))$ is \ample\ with exceptional set $S$. For any $\bm{x} \in \per(\bbZ,Y)$ of minimal period $p$, any $\varepsilon>0$, and any compact $K \subseteq \bbR \setminus S$, there exist constants $N_0=N_0(\bm{x}, K, \varepsilon)$ and $c_0 = c_0(\bm{x}, K, \varepsilon)>0$ such that for any $N \geq N_0$, there exists $\bm{y} \in \per(\bbZ,Y)$ of period $Np$ such that 
    \begin{equation}
        d(\bm x, \bm y) < \varepsilon,
        \quad 
        \Leb(K \cap \spectrum(H_{\bm y})) \leq e^{-c_0pN}.
    \end{equation}
\end{theorem}

\begin{theorem} \label{t:thinspecBcontinuum}
    Assume $S \subseteq \bbR$ is discrete, $Y \subseteq L^2([0,a))$ is perfect,  $Z \subseteq \bbR^\times $ is compact, and that $ \lambda Y$ is \ample\ with exceptional set $S$ for every $\lambda  \in Z$. 
    For any $\bm{x} \in \per(\bbZ,Y)$ of minimal period $p$, any $\varepsilon>0$, and any compact $K \subseteq \bbR \setminus S$, there exist constants $N_0=N_0(\bm{x}, K, \varepsilon,Z)$ and $c_0 = c_0(\bm{x}, K, \varepsilon,Z)>0$ such that for any $N \geq N_0$, there exists $\bm{y} \in \per(\bbZ,Y)$ of period $Np$ such that 
    \begin{equation}
        d(\bm x, \bm y) < \varepsilon,
        \quad 
        \Leb(K \cap \spectrum (H_{\lambda \bm y})) \leq e^{-c_0 N}
    \end{equation}
    for all $\lambda \in Z$.
\end{theorem}

The proofs of these theorems run identically to those of  Theorems~\ref{t:thinspecA} and \ref{t:thinspecB}.
The main difference is that one needs to replace \eqref{eq:dkdEviaconjugacy} with the corresponding fact about the IDS of a continuum periodic operator, which can be found in \cite[Lemma~2.1]{DamFilLuk2017JST} (see also \cite[eq.\ (17)]{Avila2015JAMS}).
Since the average in question is continuous rather than discrete, one needs to be able to perturb the cyclic permutation while preserving the bounds. This can be achieved with Sobolev-type inequalities and was already addressed in \cite{DamFilLuk2017JST}.
With these approximation results in hand, the proofs of Theorems~\ref{t:richToZHDcontinuum} and \ref{t:richToZHDcouplingcontinuum} proceed in the same manner as the proofs of Theorems~\ref{t:richToZHD} and \ref{t:richToZHDcoupling}.

\bibliographystyle{abbrv}

\bibliography{REUbib}

\end{document}